\documentclass[12pt,final]{amsart}
\usepackage{amssymb,euscript}
\usepackage{amscd}
\usepackage[notref,notcite]{showkeys}
\usepackage{latexsym,todonotes}
\usepackage{epsfig,pxfonts}
\usepackage{amsfonts}

\usepackage{xr,xr-hyper}
\usepackage{bbm,ifpdf}
\ifpdf
  \usepackage{pdfsync,hyperref}
\fi
\usepackage[margin=3cm,footskip=25pt,headheight=21pt]{geometry}
\usepackage{fancyhdr,mathrsfs,nicefrac}
\newtheorem{theorem}{Theorem}%[section]
\newtheorem{lemma}[theorem]{Lemma}
\newtheorem{proposition}[theorem]{Proposition}

\newtheorem{itheorem}{Theorem}

{
\theoremstyle{plain}
\newtheorem{definition}[theorem]{Definition}

\newtheorem{remark}[theorem]{Remark}
}

\newcommand{\nc}{\newcommand}
\nc{\cat}{\mathcal{V}}

\newcommand{\arxiv}[1]{\href{http://arxiv.org/abs/#1}{\tt arXiv:\nolinkurl{#1}}}

\newcommand{\id}{\operatorname{id}}

\renewcommand{\dim}{\operatorname{dim}}

\newcommand{\rk}{\operatorname{rk}}

\newcommand{\Bi}{\mathbf{i}}
\newcommand{\Bd}{\mathbf{d}}

\nc{\eE}{\EuScript{E}}
\nc{\eF}{\EuScript{F}}

\newcommand{\K}{\mathbf{k}}

\newcommand{\Z}{\mathbb{Z}}

\newcommand{\R}{\mathbb{R}}

\newcommand{\C}{\mathbb{C}}

\newcommand{\bd}{\partial}

\newcommand{\la}{\leftarrow}

\nc{\Bv}{\mathbf{v}}
  \nc{\Bw}{\mathbf{w}}
\nc{\coho}{\EuScript{G}}
\nc{\sllhat}{\mathfrak{\widehat{sl}}_\ell}
\nc{\slehat}{\mathfrak{\widehat{sl}}_e}
\nc{\glehat}{\mathfrak{\widehat{gl}}_e}

\renewcommand{\la}{\lambda}

\newcommand{\al}{\alpha}

\newcommand{\Hom}{\operatorname{Hom}}
\newcommand{\RHom}{\R\operatorname{Hom}}

\nc{\tU}{\mathcal{U}}
\nc{\eU}{\EuScript{U}}

\renewcommand{\bd}{{\mathbf{d}}}

\nc{\lift}{\gamma}

\newcommand{\cO}{\mathcal{O}}

\newcommand{\Ext}{\operatorname{Ext}}

\newcommand{\excise}[1]{}

\newcommand{\End}{\operatorname{End}}

\nc{\dbU}{\dot{\bf U}}
\nc{\dbB}{\dot{\bf B}}
\nc{\Fl}{\operatorname{Fl}}

\newcommand{\thetitle}{Comparison of canonical bases for Schur
  and universal enveloping algebras}

\begin{document}

\renewcommand{\theitheorem}{\Alph{itheorem}}
\usetikzlibrary{decorations.pathreplacing,backgrounds,decorations.markings}
\tikzset{wei/.style={draw=red,double=red!40!white,double
    distance=1.5pt,thin}}
\tikzset{dir/.style={postaction={decorate,decoration={markings,
    mark=at position .8 with {\arrow[scale=1.3]{<}}}}}}
\tikzset{rdir/.style={postaction={decorate,decoration={markings,
    mark=at position .2 with {\arrow[scale=1.3]{>}}}}}}

\noindent {\Large \bf 
\thetitle}
\bigskip\\
{\bf Ben Webster}\footnote{Supported by the NSF under Grant DMS-1151473}\\
Department of Mathematics,  University of Virginia, Charlottesville, VA
\bigskip\\
{\small
\begin{quote}
\noindent {\em Abstract.} 
We show that the canonical bases in $\dot{U}(\mathfrak{sl}_n)$ and the
Schur algebra are compatible; in fact we extend this result to
$p$-canonical bases.  This follows immediately from a fullness result
for a functor categorifying this map.  In order to prove this result, we also
explain the connections between categorifications of the Schur algebra which arise from parity
sheaves on partial flag varieties, singular Soergel bimodules and
Khovanov and Lauda's ``flag category,'' which are of some independent interest.
\end{quote}
}
\bigskip

\section{Introduction}
\label{sec:introduction}

Numerous algebras and representations that appear in Lie theory have
bases which are called ``canonical.''  There are a variety of arguments
for the importance of these bases, but surely one of their most
desirable properties is that these bases match under natural maps of algebras.

The example we'll consider in this paper is the natural projection
$\phi\colon U_q(\mathfrak{sl}_n)\to S_q(d,n)$ from the quantized universal
enveloping algebra to the $q$-Schur algebra.  The latter algebra
actually appears more naturally as the modified quantum group $\dot
{\bf U}$ with
idempotents $1_\la$ for the different weights $\la$ added.   In this
case, the Schur algebra $S_q(d,n)$ can be thought of as the quotient
of $\dot U$ which kills $1_\la$ is $\la$ cannot be written in the form
$\la=(a_1,\dots, a_n)$ with $\sum_ia_i=d$ and $a_i\geq 0$.  

The algebra $\dot U$ is endowed with a {\bf canonical basis} $\dot {\bf{B}}$
defined by
Lusztig \cite{Lusbook}, and the Schur algebra $S_q(d,n)$ with a
{\bf canonical basis} $\dot {\bf{B}}_d$ (also called {\bf IC basis}) given by realizing it as
the algebra of $GL_d(\mathbb{F}_q)$-invariant functions on the space of pairs of flags of length $n$ (of all
possible dimensions) in the space $\mathbb{F}_q^d$, and considering
the functions attached the IC sheaves smooth along $GL_d(\mathbb{F}_p)$-orbits \cite{BLM}.

\begin{itheorem}\label{main}
  Under the map $\phi$, the set $\dot {\bf{B}}\setminus(\dot
  {\bf{B}}\cap \ker \phi)$ is sent bijectively to $\dot {\bf{B}}_d$.  
\end{itheorem}
It's worth noting that this theorem was proven in \cite{SV00}.
However, there it is submerged as a special case of a more general
theorem, and new techniques have appeared in the literature in the
time since that paper that allow us to give a more straightforward and
modern proof.    Key among these is the notion of a {\bf categorical
  action} introduced by Rouquier \cite{Rou2KM} and Khovanov and Lauda
\cite{KLIII}.  We also give a generalization of this result to include
$p$-canonical bases, which proceeds along the same lines.  We also
include several results which while familiar-sounding to an expert
reader seem not to have made a clear appearance in the literature.

The map $\phi$ has a categorification which is well-established
in the literature (though its connection to the Schur algebra is
perhaps less well-known).  For simplicity in the introduction, we'll
only use categorifications of characteristic 0.
\begin{itemize}
\item The algebra $\dbU$ is the Grothendieck group of a 2-category
  $\tU$ defined in Rouquier \cite{Rou2KM} and Khovanov and Lauda
\cite{KLIII} (these definitions were recently shown to be equivalent
by Brundan \cite{Brundandef}). Actually, we'll use a slightly modified
category $\EuScript{U}$ with the same Grothendieck group.  The indecomposable 1-morphisms
correspond to the canonical basis $\dbB$ by \cite[Th. A]{WebCB}.
\item The algebra  $S_q(d,n)$ is categorified by Khovanov and Lauda's
  flag category $\mathsf{Flag}_d$ \cite[\S 5.3]{KLIII}.   The indecomposable 1-morphisms
correspond to the canonical basis $\dbB_d$, a fact we'll establish later.
\item The map $\phi$ is categorified by a 2-functor $\Phi\colon \EuScript{U}\to
  \mathsf{Flag}_d$ defined by Khovanov and Lauda.  
\end{itemize}

This categorical perspective helps to show the match of the bases
above; it follows immediately from the fact that functor above sends
indecomposable objects to indecomposable objects, itself a consequence
of:
\begin{itheorem}\label{thm:full}
  The functor $\Phi$ is full on 2-morphisms, i.e., it induces a
  surjective map $\Hom_{\tU}(u,v)\to \Hom_{\mathsf{Flag}_d} (\Phi
  u,\Phi v)$ for all 1-morphisms $u,v$.
\end{itheorem}
Besides its interesting consequence for canonical bases, this also a
beautiful illustration of the power of a categorical approach.  While
it may be that a direct proof of Theorem \ref{main} on the level of Schur algebras exists, the author has had no
luck in finding one.

It's also worth noting that a similar fullness result is proven for
$\mathfrak{sl}_2$ by Beliakova and Lauda in \cite{BLflag}.  We will
briefly indicate how their result generalizes in this case.

\section{Background}
\label{sec:background}

Throughout, we'll fix a field $\K$ which may be of any characteristic.

\subsection{Categorified $\mathfrak{sl}_n$}
\label{sec:categ-mathfr}

One of the basic objects we'll consider is the 2-category $\tU$
categorifying $\dbU$.  Rather than give a full definition of this
category, we refer the reader to a number of papers which give this
definition, such as \cite{KLIII,Rou2KM, CaLa,Webmerged}.  We will
follow the conventions of \cite{KLIII} for simplicity, with the single
exception that we think of the objects as elements of $\Z^n$ (i.e. the
$\mathfrak{gl}_n$ weight lattice) rather than $\Z^n/\Z\cdot (
1,\dots, 1)$, the weight lattice of  
$\mathfrak{sl}_n$. 
As we
mentioned earlier, work of \cite{Brundandef} shows that while the
definitions given in other papers may not immediately look equivalent, they in fact are.
The important points of the definition are that $\tU$ is a graded strict
2-category with:
\begin{itemize}
\item objects given by the weight lattice of $\mathfrak{gl}_n$, given
  by $\Z^n$.   In this space there are distinguished elements
  $\al_i=(0,\dots, 1,-1,\dots 0)$.  
\item 1-morphisms freely generated by symbols $\eE_i\colon \la\to \la+\al_i$
  and $\eF_i\colon \la\to \la-\al_i$.  Note that if two vectors do not
  have the same sum of their entries, then there are no 1-morphisms linking them.
\item 2-morphisms given by certain diagrams which induce a
  biadjunction between $\eF_i$ and $\eE_i$ (up to grading shift), and
  an isomorphism  (up to grading shift) \[\id_{\mu}^{\mu_{i+1}}\oplus
  \eE_i\eF_i\cong \eF_i\eE_i\oplus \id_{\mu}^{\mu_{i}}\] where
  $\mu=(\mu_1,\dots, \mu_n)$.  
\end{itemize}
\begin{remark}
  As mentioned above, this is a slight variation on Khovanov and
  Lauda's category; if we consider weights whose entries have a fixed
  sum $d$, this is equivalent to Khovanov and Lauda's category for
  weights where the entries sum to $d\pmod n$ (in the weight lattice
  of $\mathfrak{sl}_n$, only this residue is well defined) via the
  natural map $\Z^n\to \Z^n/\Z\cdot \langle( 1,\dots, 1)$.
\end{remark}

This category carries a duality functor defined in \cite[\S 3.3.2]{KLIII} such that $\eE_i$ and $\eF_i$
are self-dual.
As mentioned in the introduction, this serves as a categorical avatar
of the canonical basis.  
\begin{theorem}[\mbox{\cite[Th. A]{WebCB}}]
  There is an isomorphism $K(\tU)\cong U_q(\mathfrak{sl}_n)$.  If $\K$
  has characteristic 0, then the indecomposable self-dual 1-morphisms in $\tU$
  match the basis $\dbB$.
\end{theorem}
For a general field $\K$, the classes of the self-dual indecomposables only
depend on the characteristic of the field by \cite[5.11]{WebCB};  the
indecomposables with $\K=\mathbb{F}_p$ give a new basis $\mathbf{\dot B}^{(p)}$, usually
called {\bf $p$-canonical} or {\bf orthodox}; this will agree with the
usual canonical basis for $p$ large.

It will be useful for us to consider a slightly larger category; there
is a natural inclusion of $\tU\to \tU_{\mathfrak{sl}_{n+1}}$ sending
$(\la_1,\dots, \la_n)\mapsto (\la_{1},\dots,\la_{n},0)$.  We let
$\eU$ be the 2-subcatgory given by 1-morphisms in this image, but with
2-morphisms calculated in $\tU_{\mathfrak{sl}_{n+1}}$.
In terms of diagrams, this means we include bubbles (but not open
strands) labeled by the root $\al_n$, and include the local relations
we expect.  The basis of Khovanov and Lauda given in \cite{KLIII}
shows that for 1-morphisms $u$ and $v$, $\Hom_{\eU}(u,v)\cong
\Hom_{\tU}(u,v)\otimes
\K[\circlearrowleft_{n}\! (1),\circlearrowleft_{n}\!(2),\cdots]$ for
the bubbles $\circlearrowleft_n \! (k)$ with label $n$ and degree $k$
at the far left of the diagram.
The result $\eU$ is a 2-category with
essentially the same structure as $\tU$, but with slightly enlarged
endomorphisms.  In particular:
\begin{proposition}
  For any indecomposable 1-morphism $u$ in $\tU$, its image in $\eU$
  is also indecomposable.
\end{proposition}
\begin{proof}
By \cite[8.11]{WebCB}, the algebra $\End_{\tU}(u)$ is
positively graded with only scalars in degree 0.
By the description above, $\End_{\eU}(u)$ has the same property, so
$u$ is indecomposable in $\eU$ as well.
\end{proof}

\subsection{The flag category and the Schur algebra}
\label{sec:flag-category}

In this subsection, we'll discuss the corresponding story for the
Schur algebra.  This situation is more complicated, in that the same
categorification of the Schur algebra has appeared in several guises
at different points in the literature, and some explanation is
required to explain how they are connected.

Fix an integer $d$.  Throughout, we let $G=\operatorname{GL}_d(\C)$,
and $T$ be its diagonal subgroup.
For each increasing $n-1$-tuple $0\leq d_1\leq d_2\leq \cdots \leq
d_{n-1}\leq d$, we have a corresponding flag variety $\Fl(\Bd)$ with an action of
$G$.  Following their notation, let $H_{\Bd}:=H^*(\Fl(\Bd);\K)$.  We
will also at times want to consider the the equivariant cohomology
ring $H_{\Bd}^G:=H^*_G(\Fl(\Bd);\K)$. We
let $\Bd'=0\leq d_1\leq \cdots\leq  d_{i-1}\leq d_i+1\leq\cdots \leq
d_{n-1} $ and $\Bd ^{i}=0\leq d_1\leq \cdots\leq  d_i\leq d_i+1\leq\cdots \leq
d_{n-1} $.  The cohomology ring $H_{\Bd^i}$ contains $H_{\Bd}$ and
$H_{\Bd'}$ as subrings by pullback.  Let $H_{\Bd^{+i}}$ denote the
bimodule where $H_{\Bd'}$ acts on the left and $H_{\Bd}$ on the right,
and $H_{\Bd^{-i}}$ the bimodule where we reverse these actions.
\begin{definition}
  The flag category $\mathsf{Flag}_d$ is the 2-category given by:
  \begin{itemize}
  \item an object is an $(n-1)$-tuple $\Bd$ (note: in \cite{KLIII},
    the objects are defined to be the rings themselves.  This is a
    distinction without a difference).
  \item a 1-morphism $\Bd\to\Bd'$ is an object in the subcategory of 
    $H_{\Bd'}$-$H_{\Bd}$-bimodules generated by
    bimodules of the form $H_{\Bd^{\pm i}}$ under tensor, degree
    shift, direct sum and taking of summands.
  \item a 2-morphism is a degree preserving bimodule map.
  \end{itemize}
The equivariant flag category  $\mathsf{Flag}_d^G$ is the
corresponding category with bimodules over $H_{\Bd}^G$.
\end{definition}
Note that $\mathsf{Flag}_d$ is equipped with a duality induced by the
Frobenius structure on $H_\Bd$ for each $\Bd$.

We can associate the  $\mathfrak{sl}_n$-weight $\mu_{\Bd}=(d_1,d_2-d_1,d_3-d_2,\dots,
d_{n}-d_{n-1},d-d_{n})$ to each increasing $(n-1)$-tuple $\Bd$.  The increasing property precisely guarantees
that all the entries of this weight are non-negative.  Having fixed
$d$, there is at most one $\Bd$ such that $\mu=\mu_{\Bd}$, so we can
also index flag varieties and cohomology rings with the weight
$\mu_{\Bd}$ instead.  In \cite[\S 6]{KLIII}
Khovanov and Lauda construct a functor $\Phi'\colon \tU\to
\mathsf{Flag}_d$, which sends $\mu \mapsto \Bd$ if $\mu=\mu_\Bd$ and 0
otherwise.  This functor sends \[\eE_i \mapsto
H_{\Bd^{+i}}(d_{i+1}-d_i-1)\qquad \eF_i\mapsto
H_{\Bd^{-i}}(d_i-d_{i-1}).\]  Here $M(a)$ denotes the graded module
$M$ with its grading shifted downward. Note that we have used a slightly
different grading shift from \cite[\S 6]{KLIII} which matches
better with the geometry; note that the effect is just a shift by a
fixed degree on all bimodules for $\Bd$ and $\Bd'$ fixed.

This representation is compatible with the inclusion
$\tU_{\mathfrak{sl}_n}\to \tU_{\mathfrak{sl}_{n+1}}$ in the sense that
a weight of the form $(\la_1,\dots, \la_n,0)$ will give an $n$-step flag where the last term must be the whole space for dimension
reasons.  Thus, the underlying cohomology ring is unchanged by this
inclusion, as are all the bimodules attached to the Chevalley
generators of $\tU_{\mathfrak{sl}_n}$.  Thus, this inclusion allows us
to extend the representation $\Phi'$ to a representation $\Phi$ of $\eU$.

There is a second, more geometric interpretation of this category we
should also discuss.  Let $\mathsf X=\sqcup_{\mu,\nu}
\Fl(\mu)\times \Fl(\nu)$; as with any product of a space with itself, there
is a convolution product of the constructible derived category of
$\mathsf X$ which endows this space with a monoidal structure.  We
wish to consider the subcategory of parity sheaves smooth along the
$G$-orbits, as defined by Juteau, Mautner and Williamson \cite{JMW}.
This is a monoidal subcategory
as shown in \cite[4.8]{JMW}.  If $\K$ is characteristic
0, then parity sheaves are the same as sums of shifts of simple
perverse sheaves smooth along the $G$-orbits, in which case this
follows from the decomposition theorem of Beilinson, Bernstein and Deligne.  In
2-category language:
\begin{definition}
  Let $\mathsf{Perv}_d$ be the 2-category where
  \begin{itemize}
  \item objects are weights $\mu$.
  \item 1-morphisms $\nu\to \mu$ are parity sheaves over $\K$ on
    $\Fl(\mu)\times \Fl(\nu)$ which are smooth along the $G$-orbits
    for the diagonal action. The composition of 1-morphisms \[\Hom(\mu,\nu)\times
    \Hom(\nu,\xi)\to \Hom(\mu,\xi)\] is given by
    convolution
    \[\mathscr{G}\star
    \mathscr{H}:=(p_{13})_*\big(p_{12}^*\mathscr{G}\otimes_\K p_{23}^*
    \mathscr{H}\big)\big[\dim \Fl(\mu)+\dim \Fl(\xi) \big].\]
  \item 2-morphisms are morphisms in the constructible derived
    category.
  \end{itemize}
We think of this as a graded 2-category where grading shift is the
translation $[1]$ in the derived category.  A 2-morphism of degree $m$
is thus an element of $\Ext^m$.

There is also an equivariant version $\mathsf{Perv}_d^G$ of this
category, where we consider $G$-equivariant parity sheaves.  We'll use
the notation $\mathsf{Perv}_d^{(G)}$ to indicate either one of these categories.
\end{definition}
This 2-category also has an accompanying duality, induced by Verdier
duality on parity sheaves. 

While these categories seem very different in nature, they are
actually very similar.  The category of bimodules $\mathsf{Flag}_d$ is
(to quote Soergel) the ``poor man's version'' of  $\mathsf{Perv}_d$,
following similar results of Soergel for the complete flag variety \cite{Soe90,Soe00}.
Both Theorem \ref{perv=flag} and Lemma \ref{perv=schur} seem to be
``well-known'' in the correct circles, but the author does not believe
they have appeared in the literature in this generality previously.

\begin{theorem}\label{perv=flag}
 We have an equivalence of strict graded 2-categories
 $\mathsf{Perv}_d\cong \mathsf{Flag}_d$
 (resp. $\mathsf{Perv}_d^G\cong \mathsf{Flag}_d^G$) via the 2-functor
  \[\mathscr{G}\colon\nu\to \mu \mapsto H^\bullet_{(G)}(\Fl(\mu)\times
  \Fl(\nu);\mathscr{G}) \] compatible with dualities. 
\end{theorem}
\begin{proof}
  The essential point is that $\mathsf{Flag}_d$ and the cohomology of
  parity sheaves in $\Fl(\mu)\times \Fl(\nu)$ are both
  descriptions of singular Soergel bimodules modulo the positive
  degree elements in the ring of symmetric polynomials
  $R=H^\bullet(BG)=\K[x_1,\dots,x_d]^{S_d}$. The category
  $\mathsf{Flag}_d^G$  and equivariant cohomology correspond to Soergel
  bimodules without this quotient.  Since this case is so similar, we
  will only consider the non-equivariant case. We let $S_\mu\subset S_d$
  be the subgroup of the symmetric group acting on $[1,n]$ which
preserves the flag  of subsets $\{1,\dots, d_1\}\subset \{1,\dots,
d_2\}\subset \cdots\subset  \{1,\dots, d_n\}$ and let 
  $R_\mu=\K[x_1,\dots,x_d]^{S_\mu}=H^*_{G}(\Fl(\mu))$.

By definition, a singular Soergel bimodule is a sum of shifts of
summands of a tensor product of the form
\begin{equation}
R_{\mu_1}\otimes_{R_{\nu_1}}R_{\mu_2}\otimes_{R_{\nu_2}}\cdots
\otimes_{R_{\nu_{k-1}}}R_{\mu_k},\label{eq:2}
\end{equation}
for weights with $S_{\mu_1}\subset
S_{\nu_1}\supset S_{\mu_2}\subset \cdots \subset S_{\nu_{k-1}}\supset S_{\mu_k}$.
The ring $R$
 acts diagonally on these bimodules, and as left or right modules they
 are free.  This shows that 
\[(M\otimes_{R_{\nu}} N)\otimes_R\K\cong (M\otimes_R\K)\otimes_{\bar R_\nu}(N\otimes_R\K)\]
so Khovanov and Lauda's bimodules defined in \cite[5.6]{KLIII} are
reductions of the form $M\otimes_R \K$ where $M$ is a singular Soergel
bimodule.  Furthermore, \[\Hom(\bar{B},\bar{B}')\cong
k\otimes_R\Hom(B,B'),\] so we can use \cite[7.4.1]{Wthesis} to compute
the dimension of $\Hom(\bar{B},\bar{B}')$.

It is not completely obvious that every bimodule of this form is one
of Khovanov and Lauda's since they do not include all tensor products of 
bimodules as in \eqref{eq:2}; they will only consider triples
$\mu_i,\nu_i,\mu_{i+1}$ where $s_k,s_{k+1}\notin S_{\nu_i}$ and
\begin{equation}
S_{\mu_i}=\langle s_k,S_{\nu_i}\rangle, S_{\mu_{i+1}}=\langle
s_{k+1},S_{\nu_i}\rangle. \text{ or } S_{\mu_i}=\langle s_{k+1},S_{\nu_i}\rangle, S_{\mu_{i+1}}=\langle
s_{k},S_{\nu_i}\rangle.\label{eq:3}
\end{equation}
However, there is considerable redundancy in
the description of singular Soergel bimodules given in \eqref{eq:2}.  By
\cite[5.4.2]{Wthesis}, it is only necessary to consider tensor products
corresponding to a choice of ``reduced translation sequence''  (as
defined in  \cite[\S 1.3]{Wthesis}).  In type A, a reduced translation
sequence is exactly one as in \eqref{eq:3}.  Thus, $\mathsf{Flag}_d$
contains the reduction of any singular Soergel bimodule; that is, the
category would remain unchanged if we defined 1-morphisms to be
$\bar B:=B\otimes_R\K$ for $B$ a singular Soergel module over $R_{\mu}$ and
$R_{\mu'}$.  

Now, we need to establish that the functor $H^\bullet(\Fl(\mu)\times
  \Fl(\nu);-)$ is an equivalence of categories compatible with
  convolution.  The proof of equivalence is essentially the same as
  that given by Soergel in \cite[4.2.1]{Soe00}, but in a slightly more
  general context.  The category of parity sheaves is generated by
  the pushforward $f_*\K_{BS}$ from an arbitrary generalized
  Bott-Samelson $BS$ by \cite[4.6]{JMW}.  Thus, we need only prove
  full faithfulness and commutation with convolution on these
  sheaves. Faithfulness follows from the same argument as
  \cite[3.2.6]{Soe00} (note that the argument in the paper is
  incorrect, and corrected in \cite{Soecor}).  On the other hand, using
  \cite[3.4.1]{BGS96}, for parity sheaves
  $\mathscr{F},\mathscr{G}$, we have \[\dim  \RHom(\mathscr{F},\mathscr{G})=\sum_{x\in \mathsf{X}^{T\times T}} \dim
  \mathscr{F}_{x}\cdot \dim \mathscr{G}_{x},\] where
  $\mathscr{F}_{x}$ is the stalk at the $T\times T$-fixed point $x$.
  Note that we are just taking total dimensions of the cohomology of
  these complexes; since they are concentrated in odd or even degrees,
  that's the same as the absolute value of Euler characteristic.
  Since these
  dimensions of the stalks $\mathscr{F}_{x}$ also give the multiplicities of $\Delta$- and
  $\nabla$-flags on $H^\bullet(\mathscr{F})$, we have that
\[\dim \RHom(\mathscr{F},\mathscr{G})=\dim \Hom(H^\bullet(\mathscr{F}),H^\bullet(\mathscr{G}))\]
 by
  \cite[7.4.1]{Wthesis}.  Thus, we also have fullness.  

Commutation
  with convolution follows from the compatibility between the
  hypercohomology functor and pullback and pushforward between partial
  flag varieties.  Using the notation of \cite{JMW}, we consider the projection
  $\pi_{\Bd^i}^{\bd'}\colon \Fl({\Bd^i})\to\Fl(\bd)$.  We have 
  $H^\bullet((\pi_{\Bd^i}^{\bd'}\times \id)_*(\pi_{\Bd^i}^{\bd}\times \id)^*\mathscr{F})\cong
  H_{\Bd^{+i}}\otimes_{H_{\Bd}}H^\bullet(\mathscr{F})$ by
  \cite[3.3]{WW}.  
Applying this inductively, we have an isomorphism 
\[H^\bullet((f_1)_*\K_{BS_1}\star (f_2)_*\K_{BS_2})\cong H^\bullet(f_*\K_{BS})\cong \bar{B}_{BS}\cong \bar{B}_{BS_1}\otimes
  \bar{B}_{BS_2}\] where $BS$ is the Bott-Samelson for the
  concatenation of the translation sequences giving $BS_1$ and $BS_2$.
Thus, hypercohomology is a 2-functor and thus a 2-equivalence.
\end{proof}

Thus, this equivalence allows us to define a 2-functor $\Phi_{\mathsf{P}}^{(G)}\colon
\EuScript{U}\to \mathsf{Perv}_d^{(G)}$, which sends \[\eE_i \mapsto (\pi_{\Bd'}\times
\pi_{\Bd})_*\K_{\Bd^i}[d_{i+1}-d_i-1]\qquad \eF_i\mapsto (\pi_{\Bd}\times
\pi_{\Bd'})_*\K_{\Bd^i}[d_i-d_{i-1}].\]
One can think of this as a categorification of the BLM construction.
By standard techniques in \'etale cohomology, if we let
$\mathsf{X}_\Z$ be the canonical integral form of $\mathsf X$, the category
$\mathsf{Perv}_d^{(G)}$ can be defined using $\mathsf X_K$ for any field $K$, and
will not depend on the underlying field, as long as we fix $\K$.
Taking $\K=\mathbb{Q}_\ell, K=\bar{\mathbb{F}}_p$ for $\ell$ and $p$
distinct primes, the sheaves in $\mathsf{Perv}_d$ have
canonical mixed structures of weight 0.  

\nc{\BLM}{\mathcal{B\!L\!M}}

Attached to a mixed sheaf on $\mathsf{X}_{\bar{\mathbb{F}}_p}$, we
have a corresponding function on the rational points $
\mathsf{X}({{\mathbb{F}}_p})$ given by the trace of Frobenius.  This
defines a map $\tau\colon K(\mathsf{Perv}_d^{(G)})\to
\C[\mathsf{X}({{\mathbb{F}}_p})]$.  There is also a BLM map
$\BLM\colon S_p(d,n)\to \C[\mathsf{X}({{\mathbb{F}}_p})]$ which
sends
\begin{equation}
E_i\mapsto %\mathsf{e}_i(x)=
\begin{cases}
  p^{-d_{i+1}+d_i+1} & x\in \operatorname{im}(\pi_{\Bd'}\times
\pi_{\Bd})\\
0 & x\notin \operatorname{im}(\pi_{\Bd'}\times
\pi_{\Bd})
\end{cases}\qquad F_i\mapsto 
%\mathsf{f}_i(x)=
\begin{cases}
p^{-d_i+d_{i-1}} & x\in \operatorname{im}(\pi_{\Bd}\times
\pi_{\Bd'})\\
0 & x\notin \operatorname{im}(\pi_{\Bd}\times
\pi_{\Bd'})
\end{cases}\label{eq:4}
\end{equation}
\begin{lemma}\label{Frobenius-BLM}
  The map $\tau$ is an algebra map, with
  \[\tau([\Phi_{\mathsf{P}}(\eE_i)])=\BLM(E_i)\qquad \tau([\Phi_{\mathsf{P}}(\eF_i)])=\BLM(F_i).\]
\end{lemma}
\begin{proof}
  The Grothendieck trace formula implies that this map $\tau$ sends
  convolution of sheaves to 
  convolution of functions, so that shows that $\tau$  is an algebra map. Also, by
  definition it sends the classes of the sheaves
  $(\pi_{\Bd'}\times \pi_{\Bd})_*\K_{\Bd^i}[d_{i+1}-d_i-1]$ and
  $ (\pi_{\Bd}\times \pi_{\Bd'})_*\K_{\Bd^i}[d_i-d_{i-1}]$ to the
  functions of \eqref{eq:4}; the powers of $p$ appear because we must
  twist the mixed structure to remain pure of weight 0.
\end{proof}

\begin{lemma}\label{perv=schur}
  The Grothendieck group  $K(\mathsf{Perv}_d^{(G)})\cong K(\mathsf{Flag}_d^{(G)})$ is
  isomorphic to the $q$-Schur algebra.  If $\K$ is of characteristic
  0, then the classes of self-dual indecomposable objects correspond to the
  canonical basis.
\end{lemma}
\begin{proof}
  First, by necessity, we have a map $K(\Phi)\colon \dbU\to K(\mathsf{Perv}_d^{(G)})$
  induced by the functor $\Phi$.   The latter is spanned by the
  classes $[\mathcal{E}(O)]$ of the unique indecomposable parity sheaf
  $\mathcal{E}(O)$ with support equal to the orbit closure $\bar O$ by
  \cite[4.6]{JMW}.  
As shown in the proof of \cite[4.6]{JMW}, there is a reduced
translation sequence such that the pushforward from the Bott-Samelson in
$\mathsf{Perv}_d$ is supported on $\bar O$ and contains
$\mathcal{E}(O)$ as a summand with multiplicity 1.  The class of the
pushforward is in the image of $K(\Phi)$ by definition, and
we can assume all summands other than $\mathcal{E}(O)$ are by
induction on the dimension of $O$.  Thus, the map $\dbU\to K(\mathsf{Perv}_d^{(G)})$ is surjective.

On the other hand, the kernel of $K(\Phi^{(G)})$ contains $1_\mu$ if
$\Fl(\mu)=\emptyset$.  Thus, $K(\Phi^{(G)})$ factors through the $q$-Schur
algebra, and induces a surjective map $S_q(d,n)\to
K(\mathsf{Perv}_d^{(G)})$. The dimension of these algebras coincide, since
both are equal to the number of $G$-orbits on $\mathsf{X}$.  Thus, the
map is an isomorphism.  

By Lemma
\ref{Frobenius-BLM}, this isomorphism intertwines the maps $\tau$ and
$\BLM$.  If $\K$ has characteristic 0, then the indecomposable parity sheaves
on $\mathsf{X}$ are the simple perverse sheaves.  Thus $\tau$ applied
to one of these classes is the supertraces of Frobenius on the stalks
of the IC sheaf. This is the same as the image under the map $\BLM$ of
the canonical basis, by definition.  This completes the proof.
\end{proof}

Thus, as claimed in the introduction, the category $\mathsf{Flag}_d^{(G)}$
or $\mathsf{Perv}_d^{(G)}$ provides a categorification of $S_q(d,n)$.  Note
that it if $\K$ has characteristic $p>0$, the indecomposables define a
{\bf $p$-canonical} basis $\mathbf{\dot B}_d^{(p)}$, which agrees with
the canonical basis for $p$ large.  We note that this is compatible
with the $p$-canonical basis on the Hecke algebra defined by classes
of parity sheaves on $\Fl(1,2,3,\dots, n-1)$ (see \cite[1.3.4]{JMW}).

\section{Fullness}
\label{sec:fullness}

The main result of this note is that the functor $\Phi_{\mathsf{P}}^{(G)}$ is
``surjective'' in an appropriate sense.  

\begin{theorem}\label{full}
  The functor $\Phi_{\mathsf{P}}^{(G)}\colon\eU\to \mathsf{Perv}_d^{(G)}$ is essentially
  surjective on 1-morphisms and full on 2-morphisms (that is, locally full).  That
  is, for any 1-morphisms $u\to v$, we have an induced surjection $\Hom_{\eU}(u,v)\twoheadrightarrow
  \Hom_{\mathsf{Perv}_d^{(G)}}(\Phi_{\mathsf{P}}^{(G)}
  u,\Phi_{\mathsf{P}}^{(G)} v)$.  

If $n$ is invertible in $\K$, then the restriction of
$\Phi_{\mathsf{P}}$ to $\tU$ has the same properties, but the
restriction of $\Phi_{\mathsf{P}}^{G}$ does not.
\end{theorem}
We hope the statement of this theorem illustrates why we need to use
$\eU$ instead of $\tU$; while $\tU$ will suffice over a field of large
enough characteristic, $\eU$ allows us to make a cleaner statement
covering the equivariant case as well.  Note that it seems the
characteristic requirement seems to be necessary; in \cite{BLflag},
the authors must assume $2$ is invertible to get the $\mathfrak{sl}_2$
version of this result.  

  \begin{lemma}\label{lem:full-indec}
 Let $A$ and $B$ be $\K$-linear 2-categories such that 2-morphism
 spaces are always finite dimensional.  Any 2-functor $\Phi\colon A\to B$ which is essentially surjective on 1-morphisms and full
 on 2-morphisms induces a surjective map $K(\Phi)\colon K(A)\to K(B)$
 that sends the classes in $K(A)$ of indecomposable objects not in the
 kernel of $K(\Phi)$ to the classes of indecomposable objects in
 $K(B)$.  
  \end{lemma}
  \begin{proof}
    A full functor induces a surjection
    $\End(X)/J(X) \twoheadrightarrow \End(\Phi X)/J(\Phi X)$ on the
    quotients of endomorphism rings by 
    Jacobson radical.  Thus, if $X$ is
    indecomposable, $\End(X)/J(X)$ is a division algebra, and
    $\End(\Phi X)/J(\Phi X)$ must be a division algebra again or 0.
This shows that $\Phi X$ is indecomposable or 0.

    If $X$ and $Y$ are non-isomorphic indecomposables, $\End(\Phi
    (X\oplus Y))/J(\Phi (X\oplus Y))$ must be 0, a division algebra,
    or the sum of two division algebras.  It cannot contain a $2\times
    2$ matrix algebra over either of these division algebras for
    dimension reasons, and thus we must have $\Phi X\ncong \Phi Y$
    unless both are 0.
  \end{proof}
Applying this result to the functor $\Phi_{\mathsf{P}}$ immediately gives a proof
of Theorem \ref{main}.  In fact, it gives us a
  stronger result, which extends this statement to $p$-canonical bases: \begin{theorem}\label{p-main}
 The map $\phi$ induces a bijection between $\dot {\bf{B}^{(p)}}\setminus(\dot
  {\bf{B}}^{(p)}\cap \ker \phi)$ and $\dot {\bf{B}}^{(p)}_d$.  
\end{theorem}

  We'll give two proofs of  Theorem \ref{full}. This may seem
  redundant but both have interesting generalizations in different
  directions, and thus we think both are worth including.

  \begin{proof}[Proof \# 1]
Since the natural functor $\mathsf{Perv}_d^G\to \mathsf{Perv}_d$ is
surjective, let us consider $\mathsf{Perv}_d^G$ first.
We order the dimension functions $\Bd_i$ by pointwise comparison, and
induct downward in this order.  This is the same as the usual order on
the weight lattice generated by $\mu-\al_i<\mu$.
  
We
    have that $\Fl(d,d,\dots, d)$ is a
    point, so the only non-trivial 1-morphism is the identity, and its
    endomorphisms are just the scalars.  In this case, fullness is
    clear.  This establishes the base case.
 
    Assume that we know the theorem for 1-morphisms $\mu'\to \nu'$
    where either $\mu'>\mu$ or $\nu'>\nu$.  Assume that $u$ and $v$
    are indecomposable. 
Recall that $\eU$ has a ``triangular decomposition'' into two
subcategories $\eU^+$ and $\eU^-$ generated by the $\eE_i$'s and
$\eF_i$'s respectively.
 We now prove two smaller claims:
    \begin{enumerate}
    \item if $v$ is not in the image of $\eU^-$, then
      $\Hom_{\eU}(u,v)\twoheadrightarrow \Hom_{\mathsf{Perv}_d ^{G}}(\Phi_{\mathsf{P}}^G
      u,\Phi_{\mathsf{P}}^G v)$.
    \item if $u$ is not in the image of $\eU^+$, then
      $\Hom_{\eU}(u,v)\twoheadrightarrow \Hom_{\mathsf{Perv}_d ^{G}}(\Phi_{\mathsf{P}}^G
      u,\Phi_{\mathsf{P}}^G v)$.
    \end{enumerate}
    Let us first consider (1).  If $v$ is not in the image of $\eU^-$
    then by \cite[5.12]{WebCB}, we have that $v$ is a summand of
    $\eE_iv'$ for some 1-morphism $\mu+\al_i\to \nu$; let
    $e\colon \eE_iv'\to \eE_iv'$ by an idempotent whose image is $v$,
    and $v''$ be the image of $1-e$, that is the complementary
    summand.  By assumption, we have a surjection
    \[\Hom_{\eU}(u,\eE_iv')\cong
    \Hom_{\eU}(\eF_iu,v')\twoheadrightarrow
    \Hom_{\mathsf{Perv}_d ^{G}}(\eF_i\Phi_{\mathsf{P}}^{G} u,\Phi_{\mathsf{P}}^{G} v')\cong
    \Hom_{\mathsf{Perv}_d ^{G}}(\Phi_{\mathsf{P}}^{G} u,\eE_i\Phi_{\mathsf{P}}^{G} v').\]
    With we compose this map with the idempotent $\Phi_{\mathsf{P}}^{G} e$, then we
    obtain a surjection
    $\Hom_{\eU}(u,\eE_iv')\twoheadrightarrow
    \Hom_{\mathsf{Perv}_d^G}(\Phi_{\mathsf{P}}^G u,\Phi_{\mathsf{P}}^G v)$,
    which kills $\Hom_{\eU}(u,v'')$; thus, the induced map
    $\Hom_{\eU}(u,\eE_iv'
    \Hom_{\mathsf{Perv}_d^G}(\Phi_{\mathsf{P}}^G u,\Phi_{\mathsf{P}}^G v)$
    is surjective as desired.  Claim (2) follows by a symmetric
    argument.

    Thus, it remains to establish
    $\Hom_{\eU}(u,v)\twoheadrightarrow \Hom_{\mathsf{Perv}_d^{G}}(\Phi_{\mathsf{P}}^{G}
    u,\Phi_{\mathsf{P}}^{G} v)$
    for $u$ in the image of $\eU^-$ and $v$ in the image of $\eU^+$.
    For reasons of weight, the target can only be non-zero if $\mu=\nu$ and
    $u=v=\id_\mu$.  Thus, we must prove that
    \[\Hom_{\eU}(\id_\mu,\id_\mu)\twoheadrightarrow \Hom_{\mathsf{Perv}_d^G}(\id_\mu,\id_\mu)\cong H^*_G(\Fl(\mu)) \twoheadrightarrow 
    \Hom_{\mathsf{Perv}_d^G}(\id_\mu,\id_\mu)\cong H^*(\Fl(\mu)).\]

    Recall that $H^*_G(\Fl(\mu))\cong R_\mu$  is generated by the Chern classes
    $c_k(T_i)=e_k(x_{d^{i-1}+1},\dots ,x_{d^i})$ for the tautological bundles $T_i=W^i/W^{i-1}$ on
    $\Fl(\mu)$ modulo the geometric relations that $c_k(T_i)=0$ if
    $k>\rk T_i=d_i-d_{i-1}$.  By
    \cite[(6.47)]{KLIII}, the image of $\Hom_{\eU}(\id_\mu,\id_\mu)$
    is the algebra generated by the coefficients of
    $c(T_{i+1})/c(T_{i})$.
Extending to $\eU$ gives $c(T_n)$ as the
    image of bubbles with label $n$.  Thus the coefficients of
    $c(T_i)$ for each $i$ lie in this image, and they generate.

Now consider working non-equivariantly in
$\mathsf{Perv}_d$.  To complete the proof of the theorem, we need to
show that if $n$ is invertible, then
$\Hom_{\tU}(\id_\mu,\id_\mu)$ surjects to $
\Hom_{\mathsf{Perv}_d}(\id_\mu,\id_\mu) \cong H^*(\Fl(\mu))\cong
\bar R_\mu$.  The ring $\bar R_\mu$ is isomorphic to the quotient of
$R_\mu$ with the additional relation that  $c(T_1)\cdots c(T_n)=1$.  In
this case, the image of $\Hom_{\tU}(\id_\mu,\id_\mu)$
contains \[c(T_k)^n=c(T_1)\cdots c(T_n)\prod_{i=1}^{k-1}
(c(T_{i+1})/c(T_{i}))^i\prod_{i=k}^{n-1}(c(T_{i})/c(T_{i+1}))^{n-i}.\]
Thus, if $n$ is invertible then the fractional binomial coefficents
$\binom{1/n}{m}$ exist and we can take the $n$th root of this power
series, and thus obtain $c(T_k)$.  This shows the surjectivity.
  \end{proof}

  \begin{remark}\label{rem:JH}
    While we avoided using these results for the sake of clarity, this
    argument essentially uses Rouquier's Jordan-H\"older decomposition \cite[5.8]{Rou2KM}
    of a categorical module for $\mathsf{Perv}_d$, with its action of
    $\mathfrak{sl}_n\times \mathfrak{sl}_n$ on the left and right.
    The category $\mathsf{Perv}_d$ has a filtration whose successive
    quotients are categorifications of simple modules for
    $\mathfrak{sl}_n\times \mathfrak{sl}_n$.  These quotients are generated by
    the image of $\id_\mu$ modulo the subcategory generated by weights
    $> \mu$.  Thus,  we can reduce to showing fullness on $\id_\mu$ by
    \cite[5.4]{Rou2KM}, and complete the proof as above.
  \end{remark}

  \begin{proof}[Proof \#2]
As in Proof \# 1,  we'll only consider $\mathsf{Perv}_d^G$.  In this
proof, we'll use an opposite reduction.  Instead of relying on the
    highest weight vectors, we concentrate on the special weight
    $(1,\dots, 1)$ when $d=n$.  The corresponding flag variety is the
    complete flag variety $\Fl(1,2,\dots, n-1)$, and the corresponding
    bimodules are honest Soergel bimodules.  In this case, we can exploit
    the work of Elias and Khovanov which gives a diagrammatic
    description of this category.

We note that we can always
    increase $n$ by adding a redundant step in the flag, or decrease
    $n$ by removing a step where $d^i=d^{i+1}$.  
If $d^j<j$,
    then we must have a redundancy we can remove to decrease $j$;
    thus, we may assume that $d^j\geq j$ for all $j$.  On the other
    hand, if $d^j>j$, we can add redundant flags until
    $d^j=d^{j+1}=\cdots =d^{d^j}$.  Necessarily, the next step of the flag must
    be of higher dimension.   If we think of this as a weight of
    $\mathfrak{sl}_n$, it is a positive integer $a_1$ followed by
    $a_1-1$ zeros, then $a_2$ followed by $a_2-1$ zeros, etc.  Thus,
    we may assume that $n=d$ and our weight is of this form.

We now use the claim:
    \begin{enumerate}
    \item every morphism $\mu\to \nu$ between two weights of the form
      above is a summand of a composition factoring through the weight
      $(1,\dots, 1)$.  
    \end{enumerate}
The principle we use here is very simple: if $\mu$ is a weight with
$\al_i^\vee(\mu)>0$, then the identity functor is a summand of
$\eE_i\eF_i$, by the $\mathfrak{sl}_2$ relations.  It will also
considerably simplify our computation if we note that when $\eE_i$ acts
trivially on this weight space (as is the case if $d^{i-1}=d^i$), then
$\eE_i\eF_i \cong \id^{\al_i^\vee(\mu)}$.  Thus, if $a_i=2$ in
one of the examples above, we have that on the $\mu$-weight category: \[\eE_{d^i+1}\eF_{d^i+1}\cong
\id^{\oplus 2}\oplus \eF_{d^i+1}\eE_{d^i+1}\]
similarly, if $a_i=3$, then 
\begin{equation*}
 \eE_{d^i+1}\eE_{d^i+2}\eE_{d^i+1}\eF_{d^i+1}\eF_{d^i+2}\eF_{d^i+1}\cong
\eE_{d^i+1}\eE_{d^i+2}\eF_{d^i+2}\eF_{d^i+1}^{\oplus 2}
\cong \eE_{d^i+1}\eF_{d^i+1}^{\oplus 2}
\cong \id^{\oplus 6}
\end{equation*}

Thus, if our weight
is of the form described above, applying this fact inductively, we
find that whenever $d_j=j$, we have that
\begin{multline}
  \id^{(d_{j+1}-d_j)!}\cong \eE_{d^j+1}\eE_{d^j+2}\cdots
  \eE_{d^{j+1}}\eE_{d^j+1}\eE_{d^j+2}\cdots \eE_{d^{j+1}-1}\cdots
  \eE_{d^i+1}\eE_{d^i+2}\eE_{d^i+1}\\\eF_{d^i+1}\eF_{d^i+2}\eF_{d^i+1}\cdots
  \eF_{d^{j+1}-1}\cdots \eF_{d^j+2}\eF_{d^j+1}\eF_{d^{j+1}}\cdots
  \eF_{d^j+2}\eF_{d^j+1}.\label{eq:1}
\end{multline}
Aficionados of category $\cO$ will recognize this as the principle
that translation off of and then back onto a wall is a multiple of the
identity functor.  Applying this for each $d_j=j$ gives the desired
factorization.  

Thus, it suffices to prove fullness for 1-morphisms of the form $u'\circ
u$ and $v\circ v'$ factoring through $(1,\dots,1)$ (as we argued earlier, fullness is not harmed by
replacing a module with one that it is a summand of).  By the
isomorphism 
\[ \Hom(u'\circ
u,v\circ v')\cong \Hom(v^L\circ u,v'\circ u^R)
\] we can now assume that $\mu=\nu=(1,\dots, 1)$.  In this case, 
every indecomposable module is a summand of a 1-morphism factoring
through a weight space killed by the map to the Schur algebra (in
which case, the fullness is trivial) or it is a summand of a
Bott-Samelson object $Z_{\Bi}:=\eE_{i_1}\eF_{i_1}\eE_{i_2}\eF_{i_2}\cdots \eE_{i_n}\eF_{i_n}$.
The image $\Phi_{\mathsf{P}}^G (Z_{\Bi})$ is a Bott-Samelson bimodule.  The
$\Hom$ spaces between these have been calculated by Elias and
Khovanov \cite{ElKh}.  Furthermore, in \cite[6.5-6]{MSVschur},
Mackaay, Sto\v{s}i\'c and Vaz show the surjectivity of the map $\Hom_{\eU}(u,v)\twoheadrightarrow \Hom_{\mathsf{Perv}_d^G}(\Phi_{\mathsf{P}}^G
      u,\Phi_{\mathsf{P}}^G v)$ by giving diagrams in $\eU$ which hit each of the
      Elias and Khovanov's generators.  Thus, surjectivity follows.
  \end{proof}

This fact also has a significant consequence in the theory of
categorification.    Recall that Mackaay, Sto\v{s}i\'c and Vaz
\cite{MSVschur} define a categorification $\mathcal{S}(n,d)$ of the $q$-Schur algebra as
a subquotient of $\tU$; one restricts objects to $(d_1,\dots, d_n)$
where $\sum_id_i=d$, and then sets to 0 any 1-morphism factoring through
objects with any $d_i<0$.  As before, it is really more convenient to
use $\eU$, and we let $\EuScript{S}(n,d)$ to the corresponding
quotient of this category.
\begin{proposition}
The category $ \EuScript{S}(n,d)$ is
equivalent to $\mathsf{Perv}_d^G$.
\end{proposition}
We fully expect that as a quotient of $\tU$, the 2-category $\mathcal{S}(n,d)$ is equivalent to the image of the functor  $\Phi_{\mathsf{P}} \colon 
\tU\to \mathsf{Perv}_d^G$; the proof below should work to show this,
but is made a bit more difficult by the fact that the image of
$\End_{\tU}(  \id_\mu)$ in  $H^G_{\Bd}$ is not an especially well
understood ring.
\begin{proof}
  Since $\Phi_{\mathsf{P}} $ kills the correct weights, we have a
  functor $\Phi_{\EuScript{S}}\colon \EuScript{S}(n,d)\to \mathsf{Perv}_d^G$
  which we need to show is injective.  As suggested in Remark
  \ref{rem:JH}, we can show this using Rouquier's Jordan-H\"older
  filtration, or repeating the argument of Proof \# 1.  In either
  case, we reduce to showing it suffices to check injectivity of the
  map 
\[\End_{\EuScript{S}(n,d)}(\id_\mu)\to \End_{\mathsf{Perv}_d^G}(\id_\mu,).\]
Let $\Bd$ be the dimensions of the flag corresponding to $\mu$.  The
target space is simply $H^G_{\Bd}$, the space of symmetric polynomials
for the permutation group $S_{\Bd}$.   Thus, we need only show that
the kernel of the map from $\End_{\eU}(1_\mu,1_\mu)$ to
$\End_{\mathsf{Perv}_d^G}(1_\mu,1_\mu)$ is spanned by classes which
are 0 in $\EuScript{S}(n,d)$.

Let $B_i(u)=\sum u^k\circlearrowleft_i(k)$ be the power series loaded
with bubbles $\circlearrowleft_i(k)$ with label $i$ and degree $k$.
The image of $B_i(u)$ in $H^G_\Bd$ is $c(T_{i})/c(T_{i+1})$ with the
convention that $T_{n+1}$ is the trivial bundle, so for each $i$, we
have that
$c(T_i)=C_k(u)=\prod_{k=i}^nB_k(u)$.  Thus, we need to show that the
vanishing of this product in degrees $>d_i-d_{i-1}$ is one of the
relations of $\mathcal{S}(n,d)$.  We can prove this by induction on
$i$.  If $i=n$, then the bubble $\circlearrowleft_i(k)$ will vanish
if it is not fake, which holds when  
$k>d-d_{n-1}$.  

We consider $n-i+1$ concentric nested bubbles with labels
$n,n-1,n-2,\dots, i$ with thicknesses $d-d_n, d_{n}-d_{n-1},\dots,
d_{i+1}-d_i,1$, all counterclockwise except the innermost, and with
the only dots being on the innermost.  The outside region is labeled
with the weight $(d_1,d_2-d_1,\dots, d-d_n,0)$.  We use this to
construct the power series
\[ C'_i(u)=\sum_{a=d_{i-1}-d_i-1}^\infty u^{a+d_i-d_{i-1}+1}\hspace{1mm}\tikz[very thick,baseline]{\draw[dir] (0,0) circle
  (10pt);\draw[rdir] (0,0) circle (20pt);\draw[rdir] (0,0)
  circle (40pt);  \node [circle,fill=black,label=right:{$a$}, inner
  sep=2.5pt] at (-.37,0) {}; \node at (-1,0){$\cdots$};  \node at (.5,0){$i$}; \node at (0,1){$i+1$}; \node at (1.7,0){$n$};} \]
Note that the inner region has weight $(d_1,d_2-d_1,\dots,
d_{i}-d_{i-1}+1,-1,d_{i+1}-d_i,\dots, d-d_n)$, so  if the
number of dots $a$ is non-negative, then this diagram vanishes in
$\EuScript{S}(n,d)$.

The bubble slides show that sliding the innermost bubble out gives a
gives that $C'_i(u)$ is equal to $B_i(u)$ times a power series loaded with the $n-i$  remaining nested bubbles, with
the innermost now decorated with the elementary symmetric function of
degree $b$ times $u^b$.  The result \cite[4.10]{KLMS} shows that this
is the power series $C'_{i+1}(u)$.  Thus, we have that $C'_i(u)=B_i(u)
C'_{i+1}(u)$, so by induction $C'_i(u)=C_i(u)$.  

As we noted before, this shows that the coefficients $C_i^{(a)}$ for
$a> d_{i}-d_{i-1}$ vanish in $\EuScript{S}(n,d)$.  This shows the
injectivity of the map, and thus the proof is complete.
\end{proof}

It's clear that we have a functor $\mathcal{S}(n,d+n)\cong \mathsf{Perv}_{d+n}^G\to
\mathcal{S}(n,d) \cong \mathsf{Perv}_d^G$ which is induced by the functor $\eU\to \eU$ which
sends $(\mu_1,\dots, \mu_n)\to (\mu_1-1,\dots, \mu_n-1)$ and is the
identity on 1- and 2-morphisms.  Thus, these categories form an
inverse system as in \cite[\S 3]{BLflag}.
Combining
Theorem \ref{full} with the injectivity proven in
\cite[6.16]{KLIII} shows that this local system has the same ``partial
graded locally full and faithful'' property as shown in \cite[2.12]{BLflag}.
This allows us to generalize their theorem \cite[3.2]{BLflag} with the
same proof to show that:
\begin{proposition}
  The 2-category $\eU$ is the inverse limit of the equivariant flag
  categories $\mathsf{Flag}_d^G$.  If $n$ is invertible in $\K$ then
  $\tU$ is the inverse limit of the categories  $\mathsf{Flag}_d$.\qed
\end{proposition}
Even if $n$ is not invertible, we can write $\tU$ as the inverse limit of
its image in $\mathsf{Flag}_d$ or $\mathsf{Flag}_d^G$.

 \bibliography{../gen}
\bibliographystyle{amsalpha}
\end{document}